\documentclass{amsart}
\usepackage{amsmath,amssymb,enumerate,mathtools}
\newtheorem{theorem}{Theorem}[section]

\newtheorem{conjecture}[theorem]{Conjecture}
\theoremstyle{definition}

\newcommand{\bF}{\mathbb F}

\newcommand{\bC}{\mathbb C}
\newcommand{\bK}{\mathbb K}
\newcommand{\bZ}{\mathbb Z}

\newcommand{\cB}{\mathcal{B}}

\newcommand{\cM}{\mathcal{M}}

\newcommand{\cU}{\mathcal{U}}

\newcommand{\del}{\!\setminus\!}

\usepackage[utf8]{inputenc}

\title{Almost all matroids are non-representable}
\author{Peter Nelson}

\newcommand{\norm}[1]{\left\lVert #1 \right\rVert}
\newcommand{\ceiling}[1]{\left\lceil #1 \right\rceil}
\newcommand{\ceil}[1]{\ceiling{#1}}
\newcommand{\floor}[1]{\left\lfloor #1 \right\rfloor}
\newcommand{\bx}{\mathbf{x}}
\newcommand{\bu}{\mathbf{u}}

\newcommand{\mbG}{\mathbf{G}}
\begin{document}
\thanks{This research was partially supported by an NSERC Discovery Grant}
\sloppy
\maketitle
\begin{abstract}
\vspace{-0.0cm}We prove that, as $n$ tends to infinity, the proportion of $n$-element matroids that are representable tends to zero. 
\end{abstract}

\section{Introduction}

    A \emph{matroid} is a pair $M = (E,\cB)$, where $E$ is a finite `ground' set and $\cB$ is a nonempty collection of subsets of $E$ satisfying the \emph{exchange axiom}: for all $B,B' \in \cB$ and all $e \in B'\del B$, there exists $f \in B\del B'$ such that $(B \cup \{e\})\del \{f\} \in \cB$. The sets in $\cB$ are called the \emph{bases} of $M$. It follows from the definition that all bases have the same cardinality $r$; this is the \emph{rank} of $M$. Whitney introduced matroids in [\ref{whitney}] to study the abstract properties of linear dependence; his definition was motivated by the matroids that arise from the linear dependencies occuring among the columns of a matrix. For a field $\bF$, a matroid $M = (E,\cB)$ of rank $r$ is \emph{$\bF$-representable} if there is a matrix $A \in \bF^{[r] \times E}$ such that $\cB = \{B \in \binom{E}{r}\colon A[E] \text{ is nonsingular} \}$. A \emph{representable} matroid is one that is $\bF$-representable for at least one field $\bF$. These matroids form a class that is central to matroid theory, and (in various forms) the question of which matroids are representable has been omnipresent in the literature. See Oxley [\ref{oxley}] for an introduction to the subject.
    
    Despite their importance, little work has made progress on determining the asymptotic proportion of representable matroids among all matroids. The prevailing intuition has been that representable matroids should be extremely rare; this is supported by a result due to R\'onyai et al. [\ref{rbg}], which implies for each fixed field $\bF$ that almost all matroids are not $\bF$-representable, and one due to Alon [\ref{alon}] bounding the number of $\bC$-representable matroids. As observed by Mayhew et al. [\ref{mnww}], the lack of progress on this problem is in large part due to a lack of a good model for a random matroid. Many seemingly simple questions in asymptotic matroid theory remain open, and others have only succumbed recently to  rather sophisticated modern techniques [\ref{pvdp}]. 
    
    We prove the following theorem. 
    
    \begin{theorem}\label{main}
    For $n \ge 12$, there are at most $2^{n^3/4}$ representable matroids with ground set $[n]$. 
    \end{theorem}
       Knuth [\ref{knuth}] proved that the number of matroids with ground set $[n]$ is at least $2^{\frac{1}{n}\binom{n}{n/2} - n \log n}$; the theorem thus settles in the affirmative a conjecture, erroneously claimed by Brylawski and Kelly [\ref{bk}] and later explicitly posed by Mayhew et al. [\ref{mnww}], that almost all matroids are not representable. 
       
       Theorem~\ref{main} is proved by obtaining a more general bound on the number of `zero-patterns' of a system of integer polynomials that are allowed to be mapped to any field; the argument generalises the aforementioned result of R\'{o}nyai et al. [\ref{rbg}] that gives such a bound for each particular field. 
       
       We conjecture that, up to lower-order terms in the exponent, the bound in Theorem~\ref{main} is best-possible. We discuss this and make a more precise conjecture, that would characterise almost all representable matroids, in Section~\ref{conjecturesection}.  

\section{Zero-Patterns}

 Let $[n] = \{1, \dotsc, n\}$. For a polynomial $f \in \bZ[x_1, \dotsc, x_m]$, we write $\norm{f}$ for the maximum absolute value of a coefficient of $f$, where $\norm{0} = 0$. It is clear that $\norm{f+g} \le \norm{f} + \norm{g}$, and moreover, since each monomial of degree at most $d_1 + d_2$ can be written as the product of two monomials of degree at most $d_1,d_2$ respectively in at most $\binom{d_1+d_2}{d_1}$ different ways, we also have $\norm{fg} \le \binom{\deg(f) + \deg(g)}{\deg(f)}\norm{f}\norm{g}$. 
  
For each field $\bF$, let $\varphi_{\bF}\colon \bZ \to \bF$ be the natural homomorphism, and for variables $\bx =(x_1, \dotsc, x_m)$ let $\psi_{\bF}\colon \bZ[\bx] \to  \bF[\bx]$ be the ring homomorphism that applies $\varphi_{\bF}$ to each coefficient. Let $f_1, \dotsc, f_N$ be polynomials in $\bZ[\bx]$. We say a set $S \subseteq [N]$ is \emph{realisable} with respect to the tuple $(f_1,\dotsc,f_N)$ if there is a field $\bF$ and some $\bu \in \bF^m$ for which $S = \{i \in [N]\colon (\psi_{\bF}(f_i))(\bu) \ne 0_{\bF}\}$. In the theorem below and what follows, logarithms are base-two.

\begin{theorem}\label{maintech}
	Let $c,d \in \bZ$ and let $f_1, \dotsc, f_N$ be polynomials in $\bZ[x_1, \dotsc, x_m]$ with $\deg(f_i) \le d$ and $\norm{f_i} \le c$ for all $i$. If $k$ satisfies
	\[k > \tbinom{Nd+m}{m}(\log (3k) +  N \log (c(eN)^d)),\]
	then $(f_1, \dotsc, f_N)$ have at most $k$ realisable sets. 
\end{theorem} 
\begin{proof}
	Suppose not; let $S_1, \dotsc, S_k$ be distinct realisable sets. Let $\bx = (x_1, \dotsc, x_m)$ and for each $i \in [k]$ let $g_i(\bx) = \prod_{j \in S_i} f_j(\bx)$. Now $\deg(g_i(\bx)) \le Nd$, and 
	\[\norm{g_i(\bx)} \le \prod_{j=1}^N\norm{f_j(\bx)}\binom{dj}{d} \le c^N \binom{dN}{d}^N \le (c(eN)^d)^N.\]
	Let $c' =  (c(eN)^d)^N$. For each $I \subseteq [k]$, let $g_I(\bx) = \sum_{i \in I}g_{i}(\bx)$. Clearly $\deg(g_I(\bx)) \le Nd$ and $\norm{g_I(\bx)} \le kc'$. The space of $m$-variate polynomials of degree at most $Nd$ has dimension $H = \binom{Nd + m}{m}$, so there are at most $(2kc'+1)^H \le (3kc')^H$ different possible $g_I$. Now $k > H \log (3kc')$ by assumption, so $2^k > (3kc')^H$; thus there exist distinct sets $I,I' \subseteq [k]$ so that $g_I(\bx) = g_{I'}(\bx)$. 

	We may assume that $I \cap I' = \varnothing$. Let $\ell \in I \cup I'$ be chosen so that $|S_\ell|$ is as small as possible; say $\ell \in I$. Let $\bF$ be a field and $\bu \in \bF^m$ be such that $S_\ell = \{i \in [N]\colon (\psi_\bF(f_i))(\bu) \ne 0_{\bF}\}$. The definition of $g_\ell$ implies that $(\psi_\bF(g_\ell))(\bu) = \prod_{i \in S_\ell}((\psi_\bF(f_i))(\bu)) \ne 0_{\bF}$. Let $t \in I \cup I' - \{\ell\}$. Since $|S_t| \ge |S_\ell|$ and $S_t \ne S_\ell$, there is some $j \in S_t - S_{\ell}$, and $(\psi_{\bF}(f_j))(\bu) = 0_{\bF}$ so $(\psi_{\bF}(g_t))(\bu) = \prod_{i \in S_t}(\psi_{\bF}(f_i(\bu))) = 0_{\bF}$. It follows that $(\psi_{\bF}(g_I))(\bu) = (\psi_{\bF}(g_{\ell}))(\bu) \ne 0_{\bF}$ and $(\psi_{\bF}(g_{I'}))(\bu) = 0_{\bF}$, contradicting $g_I(\bx) = g_{I'}(\bx)$. 
\end{proof}

\section{The Number of Representable Matroids}
We now bound the number of rank-$r$ representable matroids on $[n]$ for each $r$; this will imply the main theorem. \begin{theorem} 
	Let $n \ge 12$ and $0 \le r \le n$. The number of rank-$r$ representable matroids on $[n]$ is at most  $k(n,r) = 2^{r(n-r)(n - (3/2)\log n + 5)}$.
	\end{theorem}
\begin{proof}
	 The result is obvious for $r \le 1$; since $k(n,r) = k(n,n-r)$ we may assume by duality that $1 < r \le \tfrac{n}{2}$. Let $\bx = [x_{ij}]$ be an $r\times (n-r)$ matrix of indeterminates, and for each $B \in \binom{[n]}{r}$ let $p_B(\bx) \in \bZ[\bx]$ be the determinant of the $r \times r$ submatrix of $[\bx|I_r]$ with column set $B$. Clearly $\deg(p_B(\bx)) \le r$ and $\norm{p_B(\bx)} = 1$ for all $B$. 

	 Let $\cM$ be the class of rank-$r$ representable matroids on $[n]$ for which the set $[n]-[n-r]$ is a basis. For each $M \in \cM$ with set of bases $\cB \subseteq \binom{[n]}{r}$, the fact that $M$ has a representation of the form $[\bu|I_r]$ implies that $\cB$ is realisable for the polynomials $(p_B[\bx] \colon B \in \binom{[n]}{r})$.  Since every representable matroid has a basis that can be mapped by one of $\binom{n}{r}$ permutations to the set $[n]-[n-r]$, the number of rank-$r$ representable matroids on $[n]$ is at most $\binom{n}{r}|\cM| \le 2^n|\cM|$; it therefore suffices to show that $k = \ceiling{2^{-n}k(n,r)}$ satisfies the inequality in Theorem~\ref{maintech}, where $(m,c,d,N) = (r(n-r),1,r,\binom{n}{r})$. Indeed, we clearly have $\log 3k < 2^n$ so the right-hand term in the inequality is
	 \begin{align*}
	 	R &= \tbinom{Nd+m}{m}\left(\log 3k + N \log (c(eN)^d)\right) \\
		&< \binom{r\binom{n}{r} + r(n-r)}{r(n-r)}\left(2^n + r\tbinom{n}{r}(\log\tbinom{n}{r} + 2)\right)\\
		&< \binom{\tfrac{4}{e}r\tbinom{n}{r}}{r(n-r)}(2^n + n2^n(n+2))\\
		&< \left(\tfrac{4}{n-r}\tbinom{n}{r}\right)^{r(n-r)}2^n(n+1)^2\\
		&< \left(\tfrac{8}{n}\tbinom{n}{r}\right)^{r(n-r)}2^{r(n-r)}.
	\end{align*} 
	where we use $n \ge 12$ and $1 < r \le \tfrac{n}{2}$, which imply the inequalities $n -r< \tfrac{4-e}{e}\binom{n}{r}$ and $2^n(n+1)^2 < 2^{2n-4} \le 2^{r(n-r)}$. 
	 Now, since $\binom{n}{r} < \tfrac{1}{\sqrt{n}}2^n$, we have 
	 \[R < 2^{r(n-r)(n -(3/2)\log n + 4)} = 2^{-r(n-r)}k(n,r) <  2^{- n}k(n,r),\]
	 as required. 
\end{proof}

It is easily shown that $(n+1)k(n,r) \le (n+1)k(n,n/2) \le 2^{n^3/4}$ for all $n \ge 12$ and $0 \le r \le n$; summing over all $r$ gives Theorem~\ref{main}. (In fact, the theorem gives an asymptotically better bound of $2^{n^3/4 - (3/8)n^2\log n + O(n^2)}$.)
	
	Applying the same argument but bounding $\binom{n}{r}$ by $\left(\tfrac{en}{r}\right)^r$ instead, one can obtain an alternative upper bound of $k'(n,r) = 2^{nr^2 \log (ne/r)}$ for large $n$; for $r = o(n/\log n)$, this bound is better than $k(n,r)$, and improves one due to Alon [\ref{alon}] on the number of $\bC$-representable matroids of a given small rank.  

\section{Random Representable Matroids}\label{conjecturesection}

The bound in Theorem~\ref{main} is essentially the maximum over all $r \in \{0, \dotsc, n\}$ of the bound of around $2^{nr(n-r)}$ obtained in Theorem~\ref{maintech}. This maximum is acheived when $r \approx n/2$; indeed, it seems very likely that almost all representable matroids on $[n]$ have rank $r \in \{\floor{\tfrac{n}{2}},\ceiling{\tfrac{n}{2}}\}$. For $n \in \bZ$, let $\delta(n) \in \{0,1\}$ be the remainder of $n$ on division by $2$. A \emph{nonbasis} of a rank-$r$ matroid on $[n]$ is a set in $\binom{[n]}{r}$ that is not a basis. For a class $\cM$ of matroids, a property holds for \emph{asymptotically almost all matroids in $\cM$} if the proportion of matroids on $[n]$ in $\cM$ with the property tends to $1$ as $n \to \infty$.

For each integer $n$,  let $d(n) = \left(\floor{\tfrac{n}{2}}-1\right)\left(\ceil{\tfrac{n}{2}}-1\right) = \tfrac{1}{4}(n^2-\delta(n)) - n + 1$. Let $\bK$ be an algebraically closed field and let $r \in  \{\floor{\tfrac{n}{2}},\ceiling{\tfrac{n}{2}}\}$. Modulo row operations and column scaling, a $\bK$-representation of a rank-$r$  matroid on $[n]$ is precisely an element of the Grassmannian $\mbG$ of all $(r-1)$-dimensional subspaces of the projective space $P(\bK^{[n]})$ modulo column scaling, and we have $\dim(\mbG) = (r-1)(n-r-1) = d(n)$. Let $\cU \subseteq \binom{[n]}{r}$ be chosen uniformly at random so that $|\cU| = d(n)-1$. We argue speculatively that $\cU$ should with high probability be the set of nonbases of a $\bK$-representable matroid. Indeed, since $|\cU| < d(n)$, the variety $A \subseteq \mbG$ of representations in which each $U \in \cU$ is a nonbasis has algebraic dimension at least $d(n) - |\cU| = 1$, and one would expect that for nearly all $\cU$, a `generic' point in $A$ does not force a point in $\binom{[n]}{r} - \cU$ to be a nonbasis, and thus should give a $\bK$-representation of a matroid whose nonbases are precisely the sets in $\cU$. Less formally, $d(n)-1$ is the maximum number of nonbases a matroid on $[n]$ can have before a $\bK$-representation requires a `coincidence'. For odd $n$ there are two possible ranks, and for even $n$ there is one; this suggests the following conjecture.  

       \begin{conjecture}	
		Asymptotically almost all representable matroids on $[n]$ have rank in $\{\floor{\tfrac{n}{2}},\ceil{\tfrac{n}{2}}\}$ and have exactly $d(n) - 1$ nonbases, and asymptotically almost all matroids on $[n]$ with rank in $\{\floor{\tfrac{n}{2}},\ceil{\tfrac{n}{2}}\}$ and with exactly $d(n)-1$ nonbases are representable. Furthermore, the number of representable matroids on $[n]$ is \[(1+\delta(n) + o(1))\binom{\tbinom{n}{\floor{n/2}}}{d(n)-1}.\] 
       \end{conjecture}
       
       Less precisely (using the estimates $\log \binom{n}{\floor{n/2}} = n - \tfrac{1}{2}\log n+ O(1)$ and $\log \binom{m}{r} = r (\log(m/r) + O_m(1))$),  this value is $2^{n^3/4 - (5/8)n^2\log n - O(n^2)}$, which matches the upper bound in Theorem~\ref{main} up to lower-order terms in the exponent. 
\section{Acknowledgements}

I thank Jason Bell, Jim Geelen, Steven Karp, and Rudi Pendavingh for helpful discussions, as well as the referee for their useful suggestions. 

\section*{References}

\newcounter{refs}

\begin{list}{[\arabic{refs}]}
{\usecounter{refs}\setlength{\leftmargin}{10mm}\setlength{\itemsep}{0mm}}
\item\label{alon}
N. Alon,
The number of polytopes, configurations and real matroids,
Mathematika 33 (1986), 62--71.

\item\label{bk}
T. Brylawski and D. Kelly,
Matroids and combinatorial geometries,
 University of North Carolina Department of Mathematics, Chapel Hill, N.C. (1980). Carolina Lecture Series.

\item\label{knuth}
D. Knuth, 
The asymptotic number of geometries, 
J. Combin. Theory. Ser. A 16 (1974), 398--400.

\item \label{mnww}
D. Mayhew, M. Newman, D. Welsh, and G. Whittle,
On the asymptotic proportion of connected matroids,
European J. Combin. 32 (2011), 882--890.

\item\label{oxley}
J. Oxley, 
Matroid Theory (second edition), 
Oxford University Press, 2011. 

\item\label{pvdp}
R. Pendavingh, J. van der Pol,
On the number of bases of almost all matroids,  
arXiv:1602.04763v3 [math.CO].

\item\label{rbg}
L. R\'onyai, L. Babai and M. K. Ganapathy,
On the number of zero-patterns of a sequence of polynomials,
J. Amer. Math. Soc. 14 (2001), 717--735.

\item\label{whitney}
H. Whitney, 
On the abstract properties of linear dependence,
Amer. J. Math. 57 (1935), 509--533. 

\end{list}

\end{document}